\documentclass{amsart}
\newcommand{\vare}{\varepsilon}

\newcommand{\lgra}{\longrightarrow}

\newcommand{\R}{\mathbb R}%
\newcommand{\T}{\mathbb T}%
\newcommand{\wit}{\widetilde}

\newtheorem{defn}{Definition}[section]
\newtheorem{thm}[defn]{Theorem}
\newtheorem{lem}[defn]{Lemma}
\newtheorem{prop}[defn]{Proposition}
\newtheorem{cor}[defn]{Corollary}
\newtheorem{ex}[defn]{Example}
\newtheorem{rem}[defn]{Remark}

\begin{document}
\title{PARTIAL ISOMETRIES OF A SUB-RIEMANNIAN MANIFOLD}

\author{MAHUYA DATTA}

\address{Statistics and Mathematics Unit, Indian Statistical Institute\\ 203,
B.T. Road, Calcutta 700108, India.\\ e-mail:
mahuya@isical.ac.in}

\maketitle

\begin{abstract} In this article, we obtain the following generalisation of isometric $C^1$-immersion theorem of Nash and Kuiper.
Let $M$ be a smooth manifold of dimension $m$ and $H$ a rank $k$ subbundle of the tangent bundle $TM$ with a Riemannian metric $g_H$. Then the pair $(H,g_H)$ defines a sub-Riemannian structure on $M$. We call a $C^1$-map $f:(M,H,g_H)\to (N,h)$ into a Riemannian manifold $(N,h)$ a {\em partial isometry\/} if the derivative map $df$ restricted to $H$ is isometric, that is if $f^*h|_H=g_H$. We prove that if $f_0:M\to N$ is a smooth map such that $df_0|_H$ is a bundle monomorphism and $f_0^*h|_H<g_H$, then $f_0$ can be homotoped to a $C^1$-map $f:M\to N$ which is a partial isometry, provided $\dim N>k$. As a consequence of this result, we obtain that every sub-Riemannian manifold $(M,H,g_H)$ admits a partial isometry in $\R^n$, provided $n\geq m+k$.\\
\mbox{}\\
Key words: Sub-Riemannian manifold, partial isometry, convex integration.\\
\mbox{}\\
Mathematics Subject Classification 2000: 53C17, 58J99.\end{abstract}

\section{Introduction}

Let $(M,g)$ be a Riemannian manifold and $f_0:M\to \R^n$ be a $C^\infty$ map such that $f_0^*h<g$ (that is, $g-f_0^*h$ is positive definite), where $h$ is the canonical metric on the Euclidean space $\R^n$. 
Nash proved in \cite{nash} that if $f_0$ is an immersion (respectively an embedding) then $f_0$ can be homotoped to an isometric immersion (respectively embedding) $f:M\to \R^n$ so that $f^*h=g$, provided $n\geq\dim M+2$. He further observed that a closed manifold $M$ that immerses (respectively embeds) in $\R^n$ also does so isometrically under the same dimension restriction. Shortly after this, Kuiper \cite{kuiper} proved that these results are  true even when $n\geq \dim M+1$. By Whitney's Immersion Theorem it is known that every manifold $M$ of dimension $m$ admits an immersion in $\R^{2m}$ and therefore, it admits an isometric $C^1$ immersion by Nash-Kuiper theorem. Isometric immersions $f:(M,g)\to (N,h)$ into any Riemannian manifold $(N,h)$ of dimension $n$ can be (locally) seen as solutions to a system of $m(m+1)/2$ equations in $n$ variables, which is clearly overdetermined when $n<m(m+1)/2$. Therefore, for sufficiently large $m$, the system remains overdetermined for $n\geq2m$. A remarkable aspect of Nash-Kuiper theorem is in showing that a overdetermined system may not only be solvable but the solution space can be `very large'.

In this paper, we obtain a generalisation of the Nash-Kuiper isometric $C^1$-immersion theorem which comes in response to certain observations of Gromov in \cite[2.4.9(B)]{gromov}. Let $M$ be a smooth manifold of dimension $m$ and $H$ a rank $k$ subbundle of the tangent bundle $TM$ with a Riemannian metric $g_H$. Then the pair $(H,g_H)$ defines a sub-Riemannian structure on $M$ \cite{bellaiche}. We call a $C^1$-map $f:(M,H,g_H)\to (N,h)$ into a Riemannian manifold $(N,h)$ a {\em partial isometry\/} if $df|_H$ is isometric, that is, if $f^*h|_H=g_H$. In the special situation, when $H$ is an integrable distribution, we obtain a regular foliation $\mathcal F$ on $M$ such that $T\mathcal F=H$. The leaves of this foliation, being integral submanifolds of $H$, inherit Riemannian structures from the metric $g_H$ on $H$. Therefore, a partial isometry in this case can be viewed as a $C^1$ map which restricts to an isometric immersion on each leaf of the foliation $\mathcal F$. 

Partial isometries are also related to Carnot-Caratheodory geometry underlying the sub-Riemannian structure $(M,H)$. 
Let $d_H$ denote the Carnot Caratheodory metric on $M$ associated with the subbundle $H$ of $TM$. Then for any two points $x,y$ of $M$, $d_H(x,y)=\infty$ if there is no $H$-horizontal path in $M$ connencting these points. Otherwise $d_H(x,y)$ is the infimum of the lengths of all $H$-horizontal paths between $x$ and $y$.  
Recall that a piecewise smooth path $\gamma:I\to M$ is called $H$-horizontal if the tangent vectors $\dot{\gamma}(t)$ lies in $H$ for all those $t\in I$ where the path is differentiable. 
Observe that a partial isometry preserves the norm of any vector in $H$, and hence preserves the lengths of $H$-horizontal paths in $M$. Thus, if $f:(M,g_H)\to (N,h)$ is a partial isometry then $f:(M,d_H)\to (N,d_h)$ is necessarily a \textit{path-isometry}, where $d_h$ is the intrinsic metric on $N$ defined by $h$.

The main result of the paper may be stated as follows:
\begin{thm} Let $M$ be a manifold with a sub-Riemannian structure $(H,g_H)$ defined as above and $f_0:M\to N$ be a $C^\infty$ map into a Riamannian manifold $(N,h)$ satisfying the following conditions:\begin{enumerate}
\item[$(i)$] The restriction of $df_0$ to the bundle $H$ is a monomorphism and 
\item[$(ii)$] $g_H-f_0^*h|_H$ is positive definite on $H$.\end{enumerate}
If $\dim N>$ {\em rank\,}$H$ then $f_0$ can be homotoped to a partial isometry $f:(M,g_H)\to (N,h)$. Furthermore, the homotopy can be made to lie in a given neighbourhood of $f_0$ in the fine $C^0$ topology. 
\label{main}\end{thm}
If we take $H=TM$ in Theorem~\ref{main} then we obtain the Nash-Kuiper isometric $C^1$-immersion theorem. 
Taking $N$ to be an Euclidean space we prove the existence of partial isometries. 
\begin{cor} Every sub-Riemannian manifold $(M,H,g_H)$ admits a partial isometry in $\R^n$ provided $n\geq \dim M+${\em rank }$H$. \label{partial isometry}\end{cor}
We also discuss several other consequences of Theorem~\ref{main} in Corollaries \ref{integrable subbundle} and ~\ref{trivial subbundle}. 

We use the convex integration technique (see \cite{gromov}, \cite{eliash}) to prove the main theorem of this paper. It would be appropriate to mention here that Gromov developed the convex integration theory on the foundation of Kuiper's technique \cite{kuiper} and applied this theory to solve many interesting problems which appear in the context of geometry. 

We organise the paper as follows. In Section 2, we outline the proof of Theorem~\ref{main}, and in Section 3 we briefly discuss the convex integration technique following the beatiful exposition of Eliashberg and Mishachev \cite{eliash}.  In Section 4 we prove the main results of the paper stated above and in Section 5 we discuss some applications of Theorem~\ref{main}.
  
\section{Sketch of the proof}
Let $(N,h)$ and $(M,H,g_H)$ be as in Section 1 and let $g_0$ be a fixed Riemannian metric on $M$ such that $g_0|_H=g_H$. 
\begin{defn} A $C^1$ map from $M$ to $N$ is called an $H$-\textit{immersion} if its derivative restricts to a monomorphism on $H$ (We have borroed this terminology from \cite{dambra-loi}). 

A $C^1$ map $f_0:M\to N$ is said to be $g_H$-\textit{short} if $g_H-f_0^*h$ restricted to $H$ is positive definite. We use the notation $f_0^*h|_H<g_H$ to express $g_H$-shortness of $f_0$. \end{defn}

It is easy to see that if $f$ is an $H$-immersion then $f^*h|_H$ is a Riemannian metric on $H$ and conversely. Also note that if $f$ is a partial isometry then it is necessarily an $H$-immersion. 

We now introduce two real-valued functions on $M$. The first function will measure the `norm' of a bilinear form on $H$ relative to a Riemannian metric on $H$. The second function will measure the `distance' between two $C^1$-functions relative to Riemannian metrics on $M$ and $N$.

Let $g$ be a Riemannian metric on $H$. For any bilinear form $\bar{g}$ we define a function
$n_g(\bar{g}):M\to \R$, as follows:

$$n_g(\bar{g})(x)=\sup_{v\in H_x\setminus 0}\frac{|\bar{g}_x(v,v)|}{g_x(v,v)}$$

Given any Riemannian metric $\tilde{g}$ on $M$ and any two $C^1$ maps $f, \bar{f}:M\to N$, define a function $d_{\tilde{g}}(f,\bar{f}):M\to \R$ by

$$d_{\tilde{g}}(f,\bar{f})(x)=\sup_{v\in T_xM\setminus 0} \frac{\|df_x(v)-d\bar{f}_x(v)\|_h}{\|v\|_{\tilde{g}}}$$
To simplify notations we shall denote $n_{g_H}$ by $n$ and $d_{g_0}$ by $d$. \\

We now outline the proof of Theorem~\ref{main}. We start with an $f_0$ as in the hypothesis of the theorem. First note that with the newly introduced terminology the hypothesis on $f_0$ reads as follows: (i) $f_0$ is an $H$-immersion, and (ii) $f_0$ is $g_H$-short. (The assumption of $g_H$-shortness is not necessary if $M$ is a closed manifold and $N$ is an Euclidean space: For, by multiplying a given $H$-immersion by a suitable positive scalar $\lambda$ we can make it $g_H$-short.) The shortness condition means that $g_H-f_0^*h$ is a Riemannian metric on $H$. As in the proof of isometric $C^1$-immersion theorem in \cite{nash}, we need a suitable decomposition of $g_H-f_0^*h|_H$ on $H$. 

\begin{lem} Let $\{U_\lambda|\lambda\in\Lambda\}$ be an open covering of the manifold $M$ such that (a) each $U_\lambda$ is a coordinate neighbourhood in $M$ and (b) for any $\lambda_0$, $U_{\lambda_0}$ intersects at most $c_1(m)$ many $U_\lambda$'s
including itself. Then  $g_H-f_0^*h|_H$ admits a decomposition as follows:
\[g_H-f_0^*h|_H=\sum_{i=1}^\infty\phi_i^2 (d\psi_i)^2|_H,\] where 
$\phi_i$ and $\psi_i$, for $i=1,2,\dots$, are $C^{\infty}$ functions on $M$ such that 
\begin{enumerate}\item[$(i)$] for each $i$ there exists a $\lambda\in\Lambda$ for which $supp\,\phi_i$ is contained in $U_\lambda$ and $d\psi_i$ is a rank 1 quadratic form on $U_\lambda$. 
\item[$(ii)$] for all but finitely many $i$, $\phi_i$ vanishes on an $U_\lambda$ and 
\item[$(iii)$] there are at most $c(m)$ many $\phi_i$'s which are non-vanshing at any point $x$.\end{enumerate}
Here $c(m)$ and $c_1(m)$ are positive integers depending on $m=\dim M$. \label{decomposition}\end{lem}
\begin{proof} Choose a subbundle $K$ of $TM$ which is complementary to $H$ and take any Riemannian metric $g_K$ on it. Then $g_M=(g_H-f_0^*h|_H)\oplus g_K+f_0^*h$ is a Riemannian metric on $TM$ which clearly restricts to $g_H$ on $H$. Moreover, $g_M-f_0^*h>0$ on $M$. Then by Nash's decomposition formula (see \cite{dambra-datta} and \cite{nash}), there exist smooth functions $\psi_i$ and $\phi_i$ as described in the lemma such that $g_M-f_0^*h=\sum_i\phi_i^2\,d\psi_i^2$. By restricting both sides to $H$ we get the desired decomposition.\end{proof}

\noindent\textbf{Construction of an Approximate solution:}  Applying Lemma~\ref{decomposition} we get a decomposition of $g_H-f_0^*h$. We then use this decomposition to obtain a $C^\infty$ map $\tilde{f}$ which is very close to a partial isometry in the sense that $n(g_H-\tilde{f}^*h)$ is sufficiently small. This is achieved following successive deformations $\bar{f}_1$, $\bar{f}_2, \dots, \bar{f}_n,\dots,$ of $f_0$ such that $\bar{f}_{i}^*h$ is approximately equal to $\bar{f}_{i-1}^*h+\phi_{i}^2d\psi_{i}^2$ for each $i$. Each step of deformation involves a convex integration (discussed in Section 3) and the deformation takes place on the open set $U_\lambda$ containing supp\,$\phi_i$ in such a way that the value of the derivative $d\bar{f}_{i-1}$ along $\tau_{i}=\ker d\psi_{i}$ is affected by a small amount.
Since at most finitely many $\phi_i$ are non-zero on any $U_\lambda$, the sequence $\{\bar{f}_i\}$ is eventually constant on each $U_\lambda$ and therefore converges to a $C^\infty$ map $\tilde{f}:M\to N$ which is very close to being a partial isometry. Indeed, for the final map $\tilde{f}$ the total error $g_H-\tilde{f}^*h$, estimated by the function $n(g_H-\tilde{f}^*h)$, can be made arbitrarily small. Moreover, $d(f_0,\tilde{f})$ can be controlled by the function $n(g_H-f_0^*h)$. See Lemma~\ref{approximation} and Lemma~\ref{recursion} for a detailed proof.\\

\noindent\textbf{Obtaining a partial isometry:}
The principal idea is to obtain a partial isometry as the limit of a sequence of $C^\infty$ smooth $g_H$-short $H$-immersions $f_j:M\longrightarrow (N,h)$ which is Cauchy in the fine $C^1$-topology and is such that the induced metric $f_j^*h$ approaches to $g_H$ on $H$ in the limit. More explicitly, the sequence $\{f_j\}$ will satisfy the following relations:
\begin{enumerate}
\item $n(g_H-f_j^*(h))\approx \frac{1}{2}n(g_H-f_{j-1}^*(h)),$

\item $d(f_j,f_{j-1})< c(m) n(g_H-f_{j-1}^*(h))^\frac{1}{2}$,
\end{enumerate}
where $c(m)$ is a constant depending on the dimension $m$ of the
manifold $M$. The $j$-th map $f_j$ can be seen as an improved approximate solution over $f_{j-1}$. The conditions $(1)$ and $(2)$ together guarantee that the sequence $\{f_n\}$ is a
Cauchy sequence in the fine $C^1$ topology and hence it converges
to some $C^1$ map $f:M\longrightarrow N$. Then the induced metric $f^*h$ must be equal to $g_H$ when restricted to $H$ by condition $(1)$. Thus $f$ is the desired partial isometry.

\section{Preliminaries of Convex Integration Theory}

In this section, we recall from \cite{gromov} and \cite{spring} the basic terminology of the theory of $h$-principle and preliminaries of convex integration theory.

Let $M$ and $N$ be smooth manifolds and $x\in M$. If $f:U\to N$ is a $C^r$ map defined on an open subset $U$ of $M$ containing $x$, then the $r$-jet of $f$ at $x$, denoted by $j^r_f(x)$, corresponds to the $r$-th degree Taylor's polynomial of $f$ relative to a coordinate system around $x$. Let $J^r(M,N)$ denote the space of $r$-jets of germs of $C^r$-maps $M\to N$ and let $p^r:J^r(M,N)\lgra M$ be the natural projection map taking $j^r_f(x)$ to $x$, which is a fibration. For any $C^r$ map $f:M\to N$, $j^r_f$ is a section of $p^r$. Moreover, if $r>s$ then there is a canonical projection $p^r_s:J^r(M,N)\to J^s(M,N)$ which takes an $r$ jet at $x$ represented by a germ $f$ to the $s$ jet of $f$ at $x$.

A \textit{partial differential relation} of order $r$ for $C^r$ maps $M\to N$ is defined as a subspace ${\mathcal R}$ of $J^r(M,N)$. If $\mathcal R$ is an open subset then we say that ${\mathcal R}$ is an {\em open\/} relation.

A section $\sigma$ of $p^r:J^r(M,N)\to M$ is said to be a \textit{section} of ${\mathcal R}$ if the image of $\sigma$ is contained in ${\mathcal R}$. A section of $\mathcal R$ is often referred as a formal solution of $\mathcal R$. If $f:M\to N$ is such that $j^r_f$ maps $M$ into $\mathcal R$ then $f$ is called a {\em solution\/} of ${\mathcal R}$.
A section $\sigma:M\to{\mathcal R}$ is called {\em holonomic\/} if $\sigma=j^r_f$ for a $C^r$-map $f:M\lgra N$.

Let $\Gamma({\mathcal R})$ denote the space of sections of the $r$-jet bundle $p^r:J^r(M,N)\longrightarrow M$ whose images lie in ${\mathcal R}$. We endow this space with the $C^0$-compact open topology. The space of $C^r$ solutions of $\mathcal R$ is denoted by Sol\,$\mathcal R$. We endow it with the $C^r$ compact-open topology. Then the $r$-th jet map $j^r:\mbox{Sol\,}{\mathcal R}\lgra \Gamma({\mathcal R})$ defined by $j^r(f)=j^r_f$ is continuous relative to the given topologies.

\begin{defn} A relation $\mathcal R$ is said to satisfy the
$h$-\textit{principle} if given a section $\sigma$ of $\mathcal R$ there
exists a solution $f$ of ${\mathcal R}$ such that $j^r_f$ is
homotopic to $\sigma$ in $\Gamma(\mathcal R)$. If the $r$-jet map $j^r$ is a weak homotopy equivalence then $\mathcal R$ is said to satisfy the \textit{parametric $h$-priniple}.

Let $\mathcal U$ be a subspace of $C^0$ maps $M\to N$. A relation $\mathcal R$ is said to satisfy the $C^0$ \textit{dense $h$-principle near} $\mathcal U$ provided given any $f\in {\mathcal U}$ and any neighbourhood $N$ of $j^0_f(M)$, every section $\sigma$ of $\mathcal R$ which lies over $j^0_f$ (i.e., $p^r_0\circ \sigma=j^0_f)$ admits a homotopy $\sigma_t$ such that $\sigma_t$ lies in $(p^r_0)^{-1}(\mathcal U)\cap{\mathcal R}$ and $\sigma_1$ is holonomic.\end{defn}

Given a differential relation $\mathcal R$, the main problem is to determine whether or not it has a solution. Proving $h$-principle is a step forward towards this goal. If a relation satisfies the $h$-principle then we can not at once say that the solution exists; however, we can conclude that if $\mathcal R$ has a section (i.e., a formal solution) then it  has a solution. Thus, we reduce a differential topological problem to a problem in algebraic topology. There are several techniques due to Gromov which address the question of $h$-principle. The convex integration theory is one such. Here we will review the convex integration theory only for first order differential relations.

Let $\tau$ be a codimension 1 integrable hyperplane distribution on $M$. Let $f$ and
$g$ be germs at $x\in X$ of two $C^1$ smooth maps from $M$ to $N$. We say
that $f$ and $g$ are $\perp$-equivalent if
$$f(x)=g(x)\ \  \mbox{and}\ \
Df_x|_{\tau}=Dg_x|_{\tau}.$$  The $\perp$-equivalence is an equivalence relation on
the $1$-jet space $J^1(M,N)$. The equivalence class of $j^1_f(x)$
is denoted by $j^\perp_f(x)$ and is called the $\perp$-jet of $f$ at $x$.
Since $\tau$ is integrable, we can choose local coordinate
systems $(U;x_1,\dots,x_{n-1},t)$ so that
$\{(x_1,\dots,x_{n-1},t):t=$\,const\} are integral submanifolds of
$\tau$. Moreover, we can express $j^1_f(x)$ as
$(j^\perp_f(x),\partial_tf(x))$, where $j^\perp_f=(\frac{\partial f}{\partial x_1},\dots,\frac{\partial f}{\partial x_{n-1}})$ and $\partial_tf$ denotes the partial derivative of $f$ in the direction of $t$.
In particular, if $M=\R^n$, $N=\R^q$ and $\tau$ is defined by the codimension one foliation $\R^{n-1}\times\R$ on $\R^n$, then the 1-jet space gets a splitting $J^1(\R^n,\R^q)=J^\perp(\R^n,\R^q)\times \R^q$.
The set of all $\perp$-jets, denoted by $J^\perp(M,N)$, has a manifold structure \cite[6.1.1]{spring} and the natural
projection map $p^{1}_\perp: J^1(M,N)\lgra J^\perp(M,N)$, taking a
$1$-jet to its $\perp$-equivalence class (relative to the given $\tau$), defines an affine bundle, in which
the fibres are affine spaces of dimension $n=\dim N$. The fibres of
this affine bundle are called {\em principal subspaces\/} relative
to $\tau$. Note that there is a unique principal subspace through each point of $J^{1}(M,N)$. In fact, the
fibre of $J^{1}(M,N)\lgra J^0(M,N)$ over any $b\in J^0(M,N)$ is
foliated by these principal subspaces and the translation map
takes principal subspaces onto principal subspaces.\\

\noindent\textbf{Notation:} We shall denote the principal subspace through $a\in J^{1}(M,N)$ by $R(a,\tau)$. If $\mathcal R$ is a first order relation and $a\in{\mathcal R}$, then the connected component of $a$ in $\mathcal R\cap R(a,\tau)$ will be denoted by ${\mathcal R}(a,\tau)$.

The following theorem, known as the $h$-Stability Theorem in the literature, (\cite[2.4.2(B)]{gromov} and \cite[Theorems
7.2, 7.17]{spring}) is a key result in the theory of convex integration.

\begin{thm} Let ${\mathcal R}$ be an open relation and let $f_0:M\lgra N$ be a $C^1$ map such
that \begin{enumerate}\item $j^\perp_{f_0}$ lifts to a section
$\sigma_0$ of $\mathcal R$ and
\item $j^1_{f_0}(x)$ lies in the convex hull of ${\mathcal
R}(\sigma_0(x),\tau_x)$ for every $x\in M$.\end{enumerate}
Let $\mathcal N$ be any neighbourhood of $j^\perp_f(M)$.
Then there exists a homotopy $\sigma_t:M\to {\mathcal R}$, $t\in [0,1]$, such that
\begin{enumerate}\item[$(i)$] $\sigma_1=j^1_{f_1}$ for some $C^1$ map $f_1:M\to N$, so that $f_1$ is a solution of $\mathcal R$ and
\item[$(ii)$] $(p^1_\perp\circ \sigma_t)(M)\subset {\mathcal N}$ for all $t\in[0,1]$. In particular $f_1$ is close to $f_0$ in the fine $C^0$ topology.\end{enumerate}

Further, if the initial map $f_0$ is a solution on $Op\,K$ for some closed set $K$ then the homotopy remains constant on $Op\,K$.\label{C-perp}
\end{thm}

\begin{rem} Since $C^\infty(M,N)$ is dense in $C^1(M,N)$ relative to the fine $C^1$-topology and $\mathcal R$ is open, we can perturb any $C^1$-solution of $\mathcal R$ to obtain a $C^\infty$ solution.
\end{rem}

\begin{defn} A connected subset $S$ in a vector space (or in an affine space) $V$ is said to be \textit{ample} if the convex hull of $S$ is all of $V$. The subset defined by the polynomial $x^2+y^2-z^2=0$ in $\R^3$ is an example of an ample subset. However, the complement of a 2-dimensional vector subspace in $\R^3$ is not ample.

A relation $\mathcal R$ is said to be \textit{ample} if for every hyperplane distribution $\tau$ on $M$, ${\mathcal R}(a,\tau)$ is ample in $R(a,\tau)$ for all $a\in {\mathcal R}$.
\end{defn}
\begin{thm}$($\cite[2.4.3, Theorem (A)]{gromov}$)$ Every open ample relation satisfies the $C^0$-dense parametric $h$-principle.\label{ample}\end{thm}

We end this section with an application of Theorem~\ref{ample} to the $H$-immersion relation; (see \cite[8.3.4]{eliash} for an alternative proof).
\begin{prop} Let $M$ be a smooth manifold and $H$ a subbundle of $TM$. Then $H$-immersions $f:M\to N$ satisfy the $C^0$-dense parametric $h$-principle provided $\dim N>$ {\em rank}\,$H$. In other words, every bundle map $(F_0,f_0):TM\to TN$ such that $F_0|_H$ is a monomorphism is homotopic through such bundle maps to an $(F,f):TM\to TN$ such that $F=df$ provided $\dim N>$ {\em rank}\,$H$.\label{H-immersion}\end{prop}

\begin{proof} The $H$-immersions $f:M\to N$ are solutions to the first order partial differential relation
\begin{center}${\mathcal R}=\{(x,y,\alpha)\in J^1(M,N)|\ \alpha|_{H_x}:H_x\to T_yN \mbox{ is injective linear} \}$.\end{center}
First of all, we prove that $\mathcal R$ is an open relation. Recall that if $(U,\phi)$ and $(V,\psi)$ are coordinate charts in $M$ and $N$ respectively then the bijection  $\tau:J^1(U,V)\to J^1(\phi(U),\psi(V)=\phi(U)\times\psi(V)\times L(\R^m,\R^n)$ defined by
\begin{center}$\tau(j^1_f(x))=(\phi(x),\psi(f(x)),d(\psi f\phi^{-1})_{\phi(x)})$
\end{center}
is a coordinate chart for the total space $J^1(M,N)$ of the 1-jet bundle \cite{guillemin}, where $m=\dim M$ and $n=\dim N$. Since $H$ is a subbundle of $TM$ we can further choose a trivialisation $\Phi:TM|_U\to U\times\R^m$ of the bundle $TM|_U$ (possibly after shrinking $U$), such that $\Phi$ maps $H$ onto $U\times \R^k$. Then $\bar{\tau}: J^1(U,V)\to \phi(U)\times\psi(V)\times L(\R^m,\R^n)$ by $\bar{\tau}(j^1_f(x))=
(\phi(x),\psi(f(x)),d(\psi f)_x\circ \bar{\Phi}^{^{-1}}_{\phi(x)})$ is a diffeomorphism, where $\bar{\Phi}=(\phi\times {Id\,})\circ \Phi:TM|_U\to \phi(U) \times \R^m$.

Now, consider the restriction morphism $r:L(\R^m,\R^n)\to L(\R^k,\R^n)$ which takes a linear transformation $L\in L(\R^m,\R^n)$ onto its restriction $L|_{\R^k}$. Let $L_k(\R^m,\R^n)$ denote  the inverse image under $r$ of the set of all monomorphisms $\R^k\to\R^n$. This is clearly an open set and it is easy to see that $\bar{\tau}$ maps $\mathcal R\cap J^1(U,V)$ diffeomorphically onto $\phi(U)\times \psi(V)\times L_k(\R^m,\R^n)$. Consequently $\mathcal R$ is open in the 1-jet space $J^1(M,N)$.

Next, we shall show that $\mathcal R$ is an ample relation. To see this, consider a codimension 1 subspace $\tau_x$ of $TM_x$ for some $x\in M$ and take a 1-jet $j^1_f(x)\in{\mathcal R}$. We need to show that the principal subspace
\begin{center} $R(j^1_f(x),\tau_x)=\{(x,f(x),\beta)\in J^1(M,N)| \beta=df_x \mbox{ on } \tau_x\}$\end{center} intersects the relation $\mathcal R$ in a pathconnected set and moreover the convex hull of the intersection, denoted by ${\mathcal R}(j^1_f(x),\tau_x)$, is all of $R(j^1_f(x),\tau_x)$. There are two possible cases:

Case 1. $H_x\subset \tau_x$. In this case, the principal subspace is completely contained in ${\mathcal R}$. Thus ${\mathcal R}(j^1_f(x),\tau_x)$ is equal to the principal subspace itself.

Case 2. $H_x\cap\tau_x$ is a codimension 1 subspace of $H_x$. Choose a vector $v\in H_x$ which is transverse to $H_x\cap \tau_x$. First observe that $R(j^1_f(x),\tau_x)$ is affine isomorphic to $T_{f(x)}N$ since any 1-jet $(x,y,\beta)$ is completely determined by $\beta(v)$. Therefore, ${\mathcal R}(j^1_f(x),\tau_x)$ is  affine equivalent to the subset
\begin{center} $S(j^1_f(x))=\{w\in T_{f(x)}N|w\not\in df_x(\tau_x\cap H_x)\}$.\end{center}
Since $\tau_x\cap H_x$ has dimension $k-1$ and $df_x$ is injective on $H_x$, the subspace $df_x(\tau_x\cap H_x)$ is of codimension at least 2 in $T_{f(x)}N$ provided $\dim N>k$. Hence $S(j^1_f(x))$ is path-connected and its convex hull is all of $T_{f(x)}N$. In other words, the convex hull of ${\mathcal R}(j^1_f(x),\tau_x)$ is all of $R(j^1_f(x),\tau_x)$.

This proves that $\mathcal R$ is an open, ample relation. Hence, we can apply Theorem~\ref{ample} to conclude that $\mathcal R$ satisfies the $C^0$-dense parametric $h$ principle.
\end{proof}

\begin{cor} Suppose that $f_0:M\to N$ is a smooth map. If $\dim N\geq \dim M+$ {\em rank} $H$, then $f_0$ can be homotoped within its $C^0$-neighbourhood to a $C^\infty$ $H$-immersion $f:M\to N$.\end{cor}
\begin{proof} In view of the above proposition it is enough to show that $f_0$ can be covered by a monomorphism $F:H\to TN$ if $\dim N\geq \dim M+$ {\em rank} $H$. It is well-known that the obstruction to the existence of such an $F$ lies in certain homotopy groups of the Stiefel manifold $V_k(\R^n)$, namely in $\pi_i(V_k(\R^n))$ for $0\leq i\leq m-1$, where $m=\dim M$, $n=\dim N$ and $k=\mbox{ rank\,}H$. Since $V_k(\R^n)$ is $n-k-1$ connected the obstructions vanish for $n\geq m+k$. This proves the corollary.
\end{proof}

\begin{rem} The set of smooth $H$-immersions $M\to N$ is an open, dense subset of $C^\infty(M,N)$ relative to the fine $C^\infty$ topology when $\dim N\geq \dim M+$ {\em rank} $H$ \cite[Proposition 2.2]{dambra-loi}.\end{rem}

\section{Proof of Theorem~\ref{main} and Corollary~\ref{partial isometry}}

Let $(N,h)$ be a smooth Riemannian manifold of dimension $n$ and $(M,H,g_H)$ be as in Section 1. Let $g_0$ be a Riemannian metric on $M$ such that $g_0|_H=g_H$. Suppose that $\dim N>$ rank $H$.

\begin{lem}$($Main Lemma$)$ Let $g$ be a Riemannian metric on $H$ such that $g<g_H$. Suppose that $f:M\to N$ is a smooth $H$-immersion and $g-f^*h=\phi^2d\psi^2$ on $H$, where $\phi$ and $\psi$ are smooth real valued functions on $M$ such that supp\,$\phi$ is contained in a coordinate neighbourhood $U$ and $d\psi$ has rank 1 on $U$. Given any two positive functions $\vare$ and $\delta$ on $M$, $f$ can be homotoped to a $C^\infty$ map $\wit{f}:M\to N$ in a given $C^0$ neighbourhood of $f$ such that $\wit{f}$ coincides with $f$ outside $U$ and satisfies the following properties:
\begin{enumerate}\item[$(i)$] $\wit{f}$ is an $H$-immersion,
\item[$(ii)$] $0\leq n_{g_H}(g-\wit{f}^*h)<\delta$,
\item[$(iii)$] $d_{g_0}(f,\wit{f}) < n_{g_H}(g-f^*h)^\frac{1}{2}+\vare$.
\end{enumerate}
\label{approximation}
\end{lem}
\begin{rem} Observe that inequality (iii) in the lemma above is independent of any particular choice of $g_0$.
\end{rem}

\begin{proof} We will prove the lemma by an application of Theorem~\ref{C-perp}. First observe that the partial isometries $(M,H,g)\to (N,h)$ are solutions to a first order differential relation $\mathcal I$ given by:
\begin{center}${\mathcal I}=\{(x,y,\alpha)\in J^1(M,N)|\,\alpha:T_xM\to T_yN$ is linear and $\alpha^*h=g$  on  $H_x$\}.\end{center}
Let $f$ be as in the hypothesis: $g-f^*h=\phi^2d\psi^2$ on $H$, where $d\psi$ is of rank 1 on $U$. Then the kernel of $d\psi$ defines a codimension 1 integrable distribution on $U$ which we will denote by $\tau$. We construct $\perp$-jet bundle $J^\perp(M,N)\to M$ relative to this hyperplane distribution $\tau$ as described in Section 3. Recall that two 1-jets $(x,y,\alpha)$ and $(x,y,\beta)$ in $J^1(M,N)$ are equivalent if $\alpha|_{\tau_x}=\beta|_{\tau_x}$, and the $\perp$-jets are equivalence classes of 1-jets under the above equivalence relation. The set of all 1-jets equivalent to $(x,y,\alpha)$ is an affine subspace of $J^1_{(x,y)}(M,N)$. This is referred as the principal subspace containing $\alpha$ and is denoted by $R(\alpha,\tau)$. We claim that for all $x\in U$
\begin{enumerate}\item[$(a)$] $R(j^1_f(x),\tau_x)\cap{\mathcal I}$ is a non-empty path-connected set.
\item[$(b)$] $j^1_f(x)$ belongs to the convex hull of $R(j^1_f(x),\tau_x)\cap{\mathcal I}$.
\item[$(c)$] There is a section $\sigma_0$ of $\mathcal I$ such that $\sigma_0(x)\in R((j^1_f(x),\tau_x)$.
\end{enumerate}

Let $V$ denote the open subset of $U$ consisting of all points $x$ such that $g-f^*h|_{H_x}\neq 0$. We shall first prove the statements (a), (b) and (c) for points in $V$. If $x\in V$ then $H_x$ is not contained in $\tau_x$ because  $g-f^*h=\phi^2d\psi^2$  on $H_x$, and therefore $H_x\cap \tau_x$ is of codimension 1 in $H_x$. Choose a smooth \textit{unit vector field} $\textbf{v}$ on the open set $V$ such that $\textbf{v}_x\in H_x$ and is $g$-\textit{orthogonal} to $H_x\cap\tau_x$ for all $x\in V$. Observe that every $\beta\in R(j^1_f(x),\tau_x)$ is then  completely determined by its value on the vector $\textbf{v}_x$. In fact, we can define an affine isomorphism $R(j^1_f(x),\tau_x)\to T_yN$ by $\beta\mapsto \beta(\textbf{v}_x)$ which maps
${\mathcal I}\cap R(j^1_f(x),\tau_x)$ onto the set
\begin{center}$S_x=\{w\in T_{f(x)}N|\  w\perp_h df_x(H_x\cap \tau_x),\ \|w\|_h=1\}$.\end{center}
Denote the rank of $H$ by $k$. Since $df_x|_{H_x}$ is injective linear, $S_x$ represents the unit sphere in a codimension $(k-1)$ subspace of $T_yN$. Hence, for $n>k$, $S_x$ is path-connected. This proves (a).

Also, $df_x(\textbf{v}_x)$ lies in the convex hull of $S_x$. Indeed, the condition $g-f^*h=\phi^2d\psi^2$ on $H_x$ implies that $h(df_x(\textbf{v}_x),df_x(w))=0$ for all $w\in H_x\cap \tau_x$ and therefore, $df_x(\textbf{v}_x)$ is $h$-orthogonal to $df_x(H_x\cap\tau_x)$. Moreover, as  $g-f^*h|_{H_x}>0$ and $f$ is an $H$-immersion it also follows that  $0<\|df_x(\textbf{v}_x)\|< 1$. Hence, $df_x(\textbf{v}_x)$ lies in the convex hull of $S_x$ proving (b).

To prove (c) note that $df_x(\textbf{v}_x)$ is orthogonal to $df_x(H_x\cap\tau_x)$. Therefore, if we define $w_0(x)={df_x(\textbf{v}_x)}/{\|df_x(\textbf{v}_x)\|_h}$ then $w_0(x)\in S_x$. Let $\sigma_0(x)$ denote the 1-jet in $R(j^1_f(x),\tau_x)\cap{\mathcal I}$ which corresponds to $w_0(x)$. Thus, $\sigma_0$ is a continuous section of $\mathcal I$ over $V$ as mentioned in (c).

If $x\in U\setminus V$, then either $\phi(x)=0$ or $d\psi_x|_{H_x}=0$ i.e., $H_x$ is contained in $\tau_x=\ker d\psi_x$. If $\phi(x)=0$ then proceeding as in the above case we can prove that (a) and (b) are true. If $H_x\subset \tau_x$ then the principal subspace $R(j^1_f(x),\tau_x)$ is completely contained in $\mathcal I$. Therefore (a) and (b) are clearly true in this case also. Further, $x\in U\setminus V$ implies that $j^1_f(x)\in{\mathcal I}$ and we can choose $\sigma_0(x)=j^1_f(x)$ on $U\setminus V$ so that (c) is proved on all of $U$. This completes the proof of the claim made above.

In fact, we have proved that the map $f:M\to N$ satisfies both (1) and (2) of the hypothesis of Theorem~\ref{C-perp} relative to the relation $\mathcal I$. Indeed, by our construction $\sigma_0$ lifts $j^\perp_f$. Further, $j^1_f(x)$ lies in the convex hull of ${\mathcal I}(\sigma_0(x),\tau_x)$ for all $x\in U$. This follows from (a) and (b) since $R(j^1_f(x),\tau_x)\cap {\mathcal I}=R(\sigma_0(x),\tau_x)\cap {\mathcal I}={\mathcal I}(\sigma_0(x),\tau_x)$ (see section 3). 
However, we cannot apply Theorem~\ref{C-perp} to $(f,\mathcal I)$, since $\mathcal I$ is not open. To surpass this difficulty, we consider a small open neighbourhood $Op\,{\mathcal I}$ of ${\mathcal I}$ in the $H$-immersion relation $\mathcal R$ and apply Theorem~\ref{C-perp} to the pair $(f,Op\,{\mathcal I})$ to obtain a smooth $H$-immersion $\tilde{f}:M\to N$ which is a solution of $Op\,{\mathcal I}$. By choosing $Op\,{\mathcal I}$ sufficiently small we can ensure that $\tilde{f}^*h|_H$ is arbitrarily close to $g$. Thus, we prove (i) and (ii) as stated in the theorem.

In order that $\tilde{f}$ satisfies condition (iii) as well, we need to modify the relation $Op\,\mathcal I$ further.
Consider the subset $S'_x=\{w\in S_x| h(w,df_x(\textbf{v}_x))\geq h(df_x(\textbf{v}_x),df_x(\textbf{v}_x)\}$ of $S_x$ (see \cite[\S 21.5]{eliash}). This is path-connected, symmetric about $w_0(x)$ and contains $df_x(\textbf{v}_x)$ in its convex hull. Moreover, for any vector $w$ in $S'_x$, $\|w-df_x(\textbf{v}_x)\| \leq \sqrt{1-\|df_x(\textbf{v}_x)\|^2}$.  Let ${\mathcal I}'$ denote the subset of ${\mathcal I}$ defined by $S'_x$, $x\in M$. Now, applying Theorem~\ref{C-perp} to $(Op\,{\mathcal I}',f)$ we obtain a $C^\infty$ map $\tilde{f}:M\to N$ which is homotopic to $f$ and is a solution of $Op\,{\mathcal I}'$. As we have already observed, $\tilde{f}$ satisfies (i) and (ii) as stated in the theorem. Further, we have,
\begin{center}$\begin{array}{rcl}\|d\tilde{f}_x(\textbf{v}_x)-df_x(\textbf{v}_x)\|_h & \leq & \sqrt{1-\|df_x(\textbf{v}_x)\|_h^2}+\vare\\
& = & \sqrt{(g-f^*h)(\textbf{v}_x,\textbf{v}_x)} + \vare
\end{array}$\end{center}
where the `error term' $\vare$ appears because of enlarging $\mathcal I$.
Since $g_0|_H=g_H$ and $\textbf{v}_x\in H$, dividing out both sides by $\|\textbf{v}_x\|_{g_0}$ we obtain from the above that
$$\frac{\|d\tilde{f}_x(\textbf{v}_x)-df_x(\textbf{v}_x)\|_h}{\|\textbf{v}_x\|_{g_0}}\leq
\sqrt{n_{g_H}(g-f^*h)}+\vare.$$
Moreover, by Theorem~\ref{C-perp} we can choose $\tilde{f}$ so that the directional derivatives of $\tilde{f}$ along $\tau$ are arbitrarily close to the corresponding derivatives of $f$. Thus we obtain that $d_{g_0}(f,\tilde{f})\leq\sqrt{n_{g_H}(g-f^*h)}+\vare$.\end{proof}

\begin{rem} In the above lemma we started with a $C^\infty$ map $f$ satisfying $f^*h<g<g_H$ and given any $\delta>0$ obtained an $\tilde{f}$ satisfying the condition $n(g-\tilde{f}^*h)<\delta$. Therefore, if we choose $\delta$ sufficiently small then $\tilde{f}$ can be made to satisfy the inequality $f^*h<\tilde{f}^*h< g_H$.\end{rem}

We now fix a countable open covering ${\mathcal U}=\{U_\lambda|\lambda\in\Lambda\}$ of the manifold $M$
which has the following properties:
\begin{enumerate}
\item[$(a)$] each $U_\lambda$ is a coordinate neighbourhood in $M$ and
\item[$(b)$] for any $\lambda_0$, $U_{\lambda_0}$ intersects atmost $c_1(m)$ many $U_\lambda$'s
including itself,
\end{enumerate}
where $c_1(m)$ is an integer depending on $m=\dim M$. This open convering will remain fixed througout. All decompositions of Riemannian metrics on $H$ will be considered with respect to this covering.

\begin{lem}Let $f_0:M\to N$ be a smooth $H$-immersion such that $f_0^*h<g_H$ on $H$. Then $f_0$ can be homotoped to a $C^\infty$ map $f_1$ in any given $C^0$ neighbourhood of $f_0$ such that
\begin{enumerate}
\item[$(i)$] ${f}_1$ is an $H$-immersion and $f_1^*h<g_H$ on $H$,
\item[$(ii)$] $0<n_{g_H}(g_H-f_1^*h)<\frac{2}{3}n_{g_H}(g_H-f_0^*h)$,
\item[$(iii)$] $d_{g_0}(f_0,f_1)<c(m) \sqrt{n_{g_H}(g_H-f_0^*h)}$,
\end{enumerate}
where $c(m)$ is a constant which depends on the dimension $m$ of $M$.\label{recursion}
\end{lem}
\begin{proof} Since $g_H-f_0^*h|_H>0$ we get a decomposition as follows:
\begin{center}$g_H-f_0^*h=2\sum_{k=1}^\infty\phi_k^2d\psi_k^2$ \ on \ $H$,\end{center}
where $\phi_k$ and $\psi_k$ are as described in Lemma~\ref{decomposition}. It further follows from the lemma that all but finitely many $\phi_i$ vanish on any $U_p$ and and at most $c(m)$ number of $\phi_i$ are non-vanishing at any point $x$. Define a sequence of Riemannian metrics on $H$ as follows: $\bar{g}_0=f_0^*h|_H$ and $\bar{g}_k=\bar{g}_{k-1}+\phi_k^2d\psi_k^2|_H$. Then each $\bar{g}_k<g_H$ and $\lim_{k\to\infty}\bar{g}_k=f_0^*h+\frac{1}{2}(g-f_0^*h)$. In successive steps we aim to increment the induced metric  on $H$ by $\phi_k^2d\psi_k^2|_H$, $k=1,2,\dots$. However, in the process of achieving this we admit an error in each step; the error in step $k$ is denoted by $\delta_k$. Thus at the end of the $k$-th step we will have a map $\bar{f}_k$ such that $\bar{f}_k^*h|_H=\bar{g}_k+\sum_{i=1}^k{\delta}_{i}$, $k=1,2,\dots$. Explicitly, we will construct a sequence of smooth maps $\{\bar{f}_k\}$, $k=1,2,\dots$, such that $\bar{f}_k^*h|_H=\bar{g}_k+\bar{\delta}_{k}$, $k=1,2,\dots$ which satisfy the following conditions:
\begin{enumerate}
\item $\bar{f}_k$ is a $g_H$-short $H$-immersion, and $\bar{f}_k=\bar{f}_{k-1}$ outside $U_k$,
\item $0\leq n_{g_H}(\bar{\delta}_k-\bar{\delta}_{k-1})<\delta'_k$,
\item $d_{g_0}(\bar{f}_{k-1},\bar{f}_k)<n_{g_H}(\bar{g}_k-\bar{g}_{k-1})^{1/2}+\vare_k$,
\item $\bar{g}_{k+1}+\bar{\delta}_{k}<g_H$.\end{enumerate}
where $\bar{f}_0=f_0$ and $\delta_0=0$ and $\sum_{k\geq 1}\vare_k<\infty$,

Taking $g=\bar{g}_1$ and $f=f_0$ in Lemma~\ref{approximation} we can prove the first step of the induction for $k=1$.
Let $\bar{f}_1^*h=\bar{g}_1+{\delta}_1$, where $\delta_1$ is such that $\bar{g}_2+\delta_1<g_H$.
Suppose we have obtained $\bar{f}_k$ satisfying (1)--(4)at the end of the $k$-th step, where we can write $\bar{f}_{k}^*h|_H=\bar{g}_{k}+\bar{\delta}_k$.
In the next step we want to increment the induced metric on $H$ by a quantity $\phi_{k+1}^2d\psi_{k+1}^2|_H$, that is, we want to induce $\bar{g}_{k+1}+\bar{\delta}_k$ on $H$. Taking $g=\bar{g}_{k+1}+\bar{\delta}_k$ and $f=\bar{f}_k$ in Lemma~\ref{approximation} we obtain a smooth map $\bar{f}_{k+1}$ which clearly satisfies (1). If we write  $\bar{f}_{k+1}^*h=\bar{g}_{k+1}+\bar{\delta}_k+\delta_{k+1}$, then $\bar{\delta}_{k+1}=\sum_{i=1}^{k+1}\delta_i$ and hence (2) and (3) are clearly satisfied. Finally, for sufficiently small choice of $\delta'_{k+1}$, we can make $\delta_{k+1}$ satisfy $\bar{g}_{k+2}+\bar{\delta}_{k+1}<g_H$. This completes the construction of $\{\bar{f}_k\}$ by induction.

Let $f_1=\lim_{k\to\infty}\bar{f}_k$. We want to show that the sequence $\{\bar{f}_k\}$ is eventually constant on any $U_\lambda$, which will imply that $f_1$ is a $C^\infty$ map. To see this take a fixed $p\in\Lambda$. It follows from Lemma~\ref{decomposition} that the set of integers
\begin{center}$\{i|\mbox{\,supp\,}\phi_i\subset U_{\lambda'} \mbox{ where } U_{\lambda'}\cap U_\lambda\neq\emptyset\}$\end{center} is finite. Let $i_\lambda$ denote the maximum element of this set. Observe that if $i>i_\lambda$ and supp $\phi_i\subset U_{\lambda'}$ then $U_\lambda\cap U_{\lambda'}=\emptyset$, so that $\phi_i$ vanishes identically on $U_\lambda$. This implies that  $\bar{f}_i=\bar{f}_{i-1}$ on $U_\lambda$ by the given construction. 
Thus, $f_1$ is a smooth map and $f_1^*h=f_0^*h+\frac{1}{2}(g-f_0^*h)+\sum_k\delta_k$. We shall prove that $f_1$ is the desired map. To see this note that
\begin{eqnarray*}n_{g_H}(g_H-f_1^*h) & = & n_{g_H}(g_H-[f_0^*h+\frac{1}{2}(g_H-f_0^*h)+\sum_k\delta_k])\\
 & = & n_{g_H}(\frac{1}{2}(g_H-f_0^*h)-\sum_k\delta_k)\\
 & \leq & \frac{1}{2}n_{g_H}(g_H-f_0^*h)+\sum_k n_{g_H}(\delta_k) \ \text{(since the sum is locally finite)}\\
 & \leq & \frac{1}{2}n_{g_H}(g_H-f_0^*h)+\sum_k \delta'_k\\
\end{eqnarray*}
We can choose $\delta'_k$ at each stage so that $\sum_k\delta'_k<\frac{1}{6}n_{g_H}(g_H-f_0^*h)$, and we obtain relation (ii). On the other hand,

\begin{eqnarray*}d_{g_0}(f_0,f_1) & \leq  & \sum_{k\geq 1} d_{g_0}(\bar{f}_{k-1},\bar{f}_{k})\\
& \leq & \sum_{k\geq 1} n_{g_H}(\bar{g}_{k}-\bar{g}_{k-1})^{1/2}+\sum_{k\geq 1}\vare_k \ \ \mbox{by (3) above}
\end{eqnarray*}
Each term of the first series on the right hand side can be estimated as follows:
\begin{eqnarray*}n_{g_H}(\bar{g}_{k+1}-\bar{g}_k)^{1/2} & \leq &
n_{g_H}(g_H-\bar{g}_k)^{1/2}\\
& = & n_{g_H}(g_H-\bar{f}_0^*h)^{1/2}.\end{eqnarray*}
However, since $\bar{g}_{k}-\bar{g}_{k-1}=\phi_{k}^2d\psi_{k}^2$ and at most $c(m)$ number of $\phi_k$ are non-vanishing at a point, the series  
$\sum_{k\geq 1} n_{g_H}(\bar{g}_{k}-\bar{g}_{k-1})^{1/2}$ is bounded above by $c(m) n_{g_H}(g_H-\bar{f}_0^*h)^{1/2}$. On the other hand, by Lemma~\ref{approximation} we are allowed to choose the sequence $\{\vare_k\}$  so that $\sum_k\vare_k<\infty$. This gives the desired relation (iii). \end{proof}

We have made all necessary preparation for the proof of Theorem~\ref{main}.
\begin{proof} \textit{of Theorem}~\ref{main}. Let $f_0$ be as in the hypothesis of the theorem. Applying Lemma~\ref{recursion} on $f_0$ recursively we can construct a sequence of $C^\infty$ maps $\{f_i:M\to N: i=1,2,\dots\}$ which has the following properties. 
\begin{enumerate}
\item $0<f_i^*h<g_H$ on $H$,
\item $0<n(g_H-f_i^*h)<\frac{2}{3}n(g_H-f_{i-1}^*h)$,
\item $d(f_{i-1},f_i)<c(m) {n(g_H-f_{i-1}^*h)}^{1/2}$ for all $i=1,2,\dots$.\end{enumerate}
It follows from (2) that the sequence of metrics $f_i^*h|_H$, $i=1,2,\dots$, coverges to $g_H$ and therefore, $\{f_i\}$ is Cauchy in the fine $C^1$ topology by (3). Hence, $\{f_i\}$ must converge to some $C^1$ map $f:M\to N$. Consequently, $\lim_{i\to \infty} f_i^*h|_H=f^*h|_H$. Thus $f^*h|_H=g_H$ and $f$ is the desired partial isometry. The homotopy between $f$ and $f_0$ is obtained by concatenating the homotopies between $f_i$ and $f_{i+1}$ for $i=1,2,\dots$.\end{proof}

\begin{rem} We have proved something stronger than what we claimed in Theorem~\ref{main}. We have proved that the partial isometries satisfy the $C^0$-dense $h$-principle in the space of $g_H$-short $H$-immersions.\end{rem}

\begin{proof}{\em of Corollary~\ref{partial isometry}}. In view of Theorem~\ref{main} it is enough to obtain a $g_H$ short $H$-immersion $f:M\to \R^n$ for $n\geq m+ k$, where $m=\dim M$ and $k=$ rank $H$. We first observe that such a map $f$ exists if $n$ is sufficiently large (see \cite{nash}). If $n\leq m+k$ then we have nothing more to show. If $n>m+k$, then we will show that there exists a vector $v\in\R^n$ such that $P_v\circ f$ is an $H$-immersion, where $P_v$ denotes the orthogonal projection of $\R^n$ onto $v^\perp$. 

We first cover $M$ by countably many open neighbouhoods $U_j$ such that $TM|_{U_j}$ is trivial. We may assume that under the trivialising map $H|_{U_j}$ sits inside $U_j\times \R^n$ as $U_j\times \R^k$. A vector $v\in \R^n$ for which $P_v\circ f$ is not an $H$-immersion on $U_j$ corresponds to a pair $(x,u)\in U_j\times \R^k$ such that $df_x(u)$ is a scalar multiple of $v$. Thus, for $v\in S^{n-1}$, $P_v\circ f$ is not an $H$-immersion on $U_j$ if and only if $v$ lies in the image of the map $F:U_j\times S^{k-1}\to S^{n-1}$ given by $(x,u)\mapsto \frac{df_x(u)}{\|df_x(u)\|}$. If $n>m+k$ then the image of this map is a set of measure zero by Sard's theorem \cite{guillemin}. Since $M$ can be covered by countably many $U_j$'s and the countable union of sets of measure zero is again a set of measure zero, we have proved that $P_v\circ f$ is an $H$-immersion for almost all $v\in \R^n$. Finally, we observe that the projection operators are length decreasing. Hence, $P_v\circ f$ is also a $g_H$ short $H$-immersion since $f$ is so.  Hence $M$ admits a $g_H$-short $H$-immersion $(M,g_H)\to (\R^n,g_{can})$ for $n\geq \dim M+$ rank $H$.\end{proof}

In the special situation, when $H$ is an integrable subbundle, we can reformulate Corollary~\ref{partial isometry} as follows.
\begin{cor} Every Riemannian manifold $(M,g_0)$ with a regular foliation $\mathcal F$ admits a $C^1$-map $f:M\to \R^n$  which restricts to an isometric immersion on each leaf of the foliation, provided $n\geq\dim M+\dim {\mathcal F}$.\label{integrable subbundle}
\end{cor} 

\begin{rem} We observed in Section 1 that a partial isometry of a sub-Riemannian manifold $(M,H,g_H)$ is also a path isometry with respect to the Carnot-Caratheodory metric $d_H$ on $M$ induced by $g_H$. Therefore, by Corollary~\ref{partial isometry} there is a path-isometry $f:(M,d_H)\to (\R^n, d_{can})$, provided $n\geq \dim M+$ rank $H$. We refer to a result in \cite[Corollary 1.5]{donne} which is of similar inerest.\end{rem}

\section{Applications of Theorem~\ref{main}}

In this section we discuss some applications of Theorem~\ref{main}. Throughout, we assume $M$ to be a closed manifold.  
First observe that if $M$ is a closed manifold and $N$ is an Euclidean space, then the hypothesis of Theorem~\ref{main} can be relaxed to conclude the existence of partial isometry. Indeed, we do not require the $g_H$-shortness condition on $f_0$; given any $H$-immersion $f_0:M\to \R^n$ we can obtain a $g_H$-short $H$-immersion $f_1$ which is of the form $\lambda f_0$, where $\lambda$ is a positive real number. Applying Theorem~\ref{main} we can then homotope $f_1$ to a partial isometry $f:M\to \R^n$. However, the resulting partial isometry cannot be made $C^0$-close to $f_0$ by this technique, since $f_1=\lambda f_0$ may not be $C^0$-close to $f_0$.

\begin{cor}(\cite{gromov}) Let $M$ be a closed manifold and $\partial_i$, $i=1,2,\dots,k$, be linearly independent vector fields on $M$. Then there exists a $C^1$ map $f:M\to\R^{k+1}$ such that $\langle \partial_if,\partial_jf\rangle=\delta_{ij}$, $1\leq i\leq j\leq k$, where $\delta_{ij}=1$ if $i=j$ and $0$ if $i< j$.\label{trivial subbundle}\end{cor}

\begin{proof} Let $H$ be the (trivial) subbundle of $TM$ spanned by the vector fields $\partial_i$, $i=1,2,\dots,k$. Define a Riemannian metric $g_H$ on $H$ by the relations
\begin{center}$g_H(\partial_i,\partial_j)=0$ if $i\neq j$ and $g_H(\partial_i,\partial_i)=1$\end{center}
for $i,j=1,\dots,k$. Consider the triple $(M,H,g_H)$ as defined above. Since $H$ is trivial, Proposition~\ref{H-immersion} guarantees the existence of an $H$-immersion $M\to \R^{k+1}$ which can be scaled appropriately in order to get a strictly $g_H$ short $H$-immersion, since the manifold $M$ is closed. Hence by Theorem~\ref{main} there exists a $C^1$ partial isometry $f:(M, g_H)\to (\R^{k+1}, g_{can})$. This means that $\langle \partial_if,\partial_jf\rangle=g_H\langle \partial_i,\partial_j\rangle$ and the proof is now complete.\end{proof}

\begin{rem} A more general form of the above result is in fact true. Let $\Sigma(k,\R)$ denote the set of all positive definite symmetric matrices over reals and let $g:M\to \Sigma(k,\R)$ be any smooth map. Then $g$ can be realised as the matrix $(\langle\partial_i f,\partial_j f\rangle)_{i,j}$ for some $C^1$-function $f:M\to \R^{k+1}$ provided $M$ is a closed manifold.\end{rem}

Gromov observed in \cite{gromov} that if we have $k=1$ in Corollary~\ref{trivial subbundle} then we can actually obtain $C^\infty$  partial isometries. We here give a direct proof of this result without going into the convex integration theory. 
\begin{thm} If $M$ is a closed manifold and $X$ is a smooth nowhere vanishing vector field on $M$, then there exists a $C^\infty$-map $f:M\to\R^2$ such that $\langle Xf,Xf\rangle=1$. \label{smooth partial isometry}\end{thm}
\begin{proof}{\em of Theorem~\ref{smooth partial isometry}} Let $X$ be a smooth vector field on $M$ which is nowhere vanishing. We need to solve the equation $<Xf,Xf>=1$, for smooth functions $f:M\to \R^2$. Let $H$ denote the 1-dimensional (integrable) distribution on $M$ determined by $X$. By Proposition~\ref{H-immersion} there exists an $H$-immersion $f_0:M\to \R^2$ which implies that $Xf_0$ is a nowhere vanishing function on $M$. Since $M$ is a closed manifold, without loss of generality we may assume that $0<\langle Xf_0,Xf_0\rangle=\phi^2<1$. This condition means that $f$ is $g_H$-short if we define $g_H$ by $g_H(X,X)=1$. Consider the equation $\langle X(f_0+\alpha),X(f_0+\alpha)\rangle=1$, where $\alpha:M\to \R^2$ is a smooth map. This reduces to
\begin{equation} \langle X\alpha,X\alpha\rangle+ 2\langle Xf_0,X\alpha\rangle=1-\phi^2.
\end{equation}
 We split this into a system of two equations as follows:
\begin{equation} \langle Xf_0,X\alpha\rangle=0\ \ \ \ \langle X\alpha,X\alpha\rangle=1-\phi^2.\end{equation}
Now note that $\beta=\frac{\sqrt{1-\phi^2}}{\|Xf_0\|}\rho_{\pi/2}\circ (Xf_0)$ is a formal solution of the above system, where $\rho_{\pi/2}$ is the rotation on $\R^2$ through the angle $\pi/2$ in the anticlockwise direction. Hence the problem of finding the desired $f$ reduces to solving the equation $X\alpha=\beta$, where $\beta:M\to\R^2$ is a nowhere vanishing smooth function.

The vector field $X$ can be considered as a first order linear differential operator on $C^\infty(M,\R^2)$. We will define a differential operator $M:C^\infty(M,\R^2)\to C^\infty(M,\R^2)$ such that $X(M(\beta))=\beta$ for all $\beta\in C^\infty(M,\R^2)$. We first observe that this problem is local and therefore it is enough to define local inversion operators on open sets around each point \cite[2.3.8]{gromov}. To see this let $\{U_\mu\}$ be a locally finite open covering of $M$ by coordinated neighbourhoods on each of which we have a local inversion $M_\mu$ of $X$. Define $M$ by $M\beta=\sum_\mu M_\mu(\phi_\mu \beta)$ where $\{\phi_\mu\}$ is a partition of unity subordinate to the open covering $\{U_\mu\}$. Then $M$ is a global inversion operator since

\begin{center}$\begin{array}{rcll} X(M\beta) & = & \sum_\mu X(M_\mu(\phi_\mu \beta)) & \mbox{(since the sum is locally finite)}\\
& = & \sum_\mu \phi_\mu\beta & \mbox{(since } \phi_\mu\beta\mbox{ is supported on } U_\mu)\\
& = & \beta\end{array}$\end{center}

It now remains to prove the local existence of an inversion $M$ of $X$. Recall that the distribution $H$ is integrable so that $M$ is foliated by integral curves of $H$. Indeed around each point of $M$ there exists a coordinate system $(U,(x,t))$ such that $U$ is homeomorphic to $\R^{n-1}\times \R$ and $\frac{\partial}{\partial t}$ is tangent to $H$, so that $X$ can be expressed as $\psi(x,t)\frac{\partial}{\partial t}$ on $U$. Therefore, the problem reduces to solving the equation $\psi(x,t)\frac{\partial\alpha}{\partial t}=\beta(x,t)$. Since $\psi(x,t)$ is nowhere vanishing we can define $M\beta$ by $\int_0^t\beta(x,t)/\psi(x,t)\,dt$. This completes the proof.\end{proof}

\begin{rem} Possibly we do not require the closedness condition on $M$ in the above two results (see \cite{gromov}).\end{rem}

We end this section with an example of $C^\infty$ partial isometry which supports the above theorem.
\begin{ex} {\em Let $\psi:\R^2\to \R^3$ be the smooth immersion defined by \[\psi(\theta,\phi)=((b+a\cos\theta)\cos\phi,(b+a\cos\theta)\cos\phi, a\sin\theta),\]
where $(\theta,\phi)\in\R^2$ and  $a$, $b$ are two real numbers with $0<a<b$. 
The image of $\psi$ is a 2-torus $\T^2$ which is a manifold with local parametrisations defined by $\psi$. Let $g$ denote the flat metric on $\T^2$ induced by $\psi$. 
 
For any real number $\alpha$, we have a 1-dimensional foliation of $\R\times \R$ by lines of slope $\alpha$. Let $\mathcal F_\alpha$ denote the image of this foliation on $\T^2$ under $\psi$, and $H_\alpha$ the corresponding 1-dimensional distribution on $\T^2$. Let $g_\alpha$ be the restriction of $g$ to $H_\alpha$. Consider the map $f:\T^2\to\R^2$ defined by $f(x,y,z)=(x,y)$. It is easy to see that $f$ is $g_\alpha$-short $H_\alpha$-immersion. In particular, if $\alpha=0$ then $f_0$ itself is a partial isometry.}
\end{ex}

\end{document}